\documentclass{article}

\usepackage{amsthm,amsfonts,amssymb,amsmath,amsxtra}
\usepackage[all]{xy}

\usepackage[colorlinks=true]{hyperref}

\usepackage[usenames]{color}

\usepackage{mathrsfs}

\theoremstyle{plain}
\newtheorem{thm}{Theorem}[section]
\newtheorem{lem}[thm]{Lemma}
\newtheorem{cor}[thm]{Corollary}

\theoremstyle{definition}

\theoremstyle{remark}

     \newcommand{\FG}{{\mathbf{G}}}

    \newcommand{\FP}{{\mathbf{P}}}
    \newcommand{\FM}{{\mathbf{M}}}
    \newcommand{\FN}{{\mathbf{N}}}

    \newcommand{\FX}{{\mathbf{X}}}

    \newcommand{\sT}{{\mathscr{T}}}

    \newcommand{\BC}{{\mathbb {C}}}

     \newcommand{\BN}{{\mathbb {N}}}
     
    \newcommand{\BQ}{{\mathbb {Q}}}

     \newcommand{\fn}{{\mathfrak{n}}}

    \newcommand{\sL}{{\mathscr{L}}}

    \newcommand{\ad}{{\mathrm{ad}}}

    \renewcommand{\d}{{\mathrm{d}}}

    \newcommand{\el}{{\mathrm{ell}}}

    \newcommand{\Gal}{{\mathrm{Gal}}} 
    
    \newcommand{\Hom}{{\mathrm{Hom}}}

    \newcommand{\Lie}{{\mathrm{Lie}}}

    \newcommand{\ob}{{\mathrm{ob}}}

     \newcommand{\R}{{\mathrm{R}}}

    \newcommand{\reg}{{\mathrm{reg}}}

    \newcommand{\pair}[1]{\langle {#1} \rangle}

    \newcommand{\incl}{\hookrightarrow}
    \newcommand{\sk}{\medskip}
    \newcommand{\lra}{\longrightarrow}
    \newcommand{\ra}{\rightarrow} 
    
    \newcommand{\bs}{\backslash}

    \newcommand{\s}{\sk\noindent}

    \newcommand{\abs}[1]{\lvert#1\rvert}

\title{An orthogonality relation for spherical characters of
supercuspidal representations}
\author{Chong Zhang}
\begin{document}
\date{}
\maketitle

\begin{abstract}
In this note, we show that, in the setting of Galois pairs, the
spherical characters of unitary supercuspidal representations
satisfy an orthogonality relation.
\end{abstract}

\section{Main result}
Let $F$ be a finite extension field of $\BQ_p$ for an odd prime $p$,
and $E$ a quadratic field extension of $F$. Let $G$ be a connected
reductive group over $F$, and $\FG=\R_{E/F}G$ the Weil restriction
of $G$ with respect to $E/F$. The nontrivial automorphism in
$\Gal(E/F)$ induces an involution $\theta$, defined over $F$, on
$\FG$. The pair $(\FG,G)$ is called a Galois pair, which is a kind
of symmetric pairs.

Let $\pi$ be an irreducible admissible unitary representation of
$G(E)=\FG(F)$. We say that $\pi$ is $G$-distinguished if the space
$\Hom_{G(F)}(\pi,\BC)$ is nonzero. Given an element $\ell$ in
$\Hom_{G(F)}(\pi,\BC)$, the {\em spherical character}
$\Phi_{\pi,\ell}$ associated to $\ell$ is the distribution on $G(E)$
defined by
$$\Phi_{\pi,\ell}(f):=\sum_{v\in\ob(\pi)}\ell(\pi(f)v)\overline{\ell(v)},
\quad f\in C_c^\infty(G(E)),$$ where $\ob(\pi)$ is an orthonormal
basis of the representation space $V_{\pi}$ of $\pi$. In this note,
our main goal is to show that spherical characters satisfy an
orthogonality relation when $\pi$ is unitary supercuspidal.

Before stating our result, we introduce some notations. Recall that
an element $g\in G(E)$ is called $\theta$-regular if
$s(g):=g^{-1}\theta(g)$ is regular semisimple in $G(E)$ in the usual
sense; a $\theta$-regular element $g$ is called $\theta$-elliptic if
the identity component of the centralizer of $s(g)$ in $G$ is an
elliptic $F$-torus. We denote by $G(E)_\reg$ (resp. $G(E)_\el$) the
subset of $\theta$-regular (resp. $\theta$-elliptic) elements of
$G(E)$.

The following fundamental theorem is proved by J. Hakim (see
\cite[Theorem 1]{ha}).

\begin{thm}[Hakim]\label{thm. hakim}
The spherical character $\Phi_{\pi,\ell}$ is locally integrable on
$G(E)$ and locally constant on the $\theta$-regular locus
$G(E)_\reg$.
\end{thm}

Fixing a Haar measure $\d g$ on $G(E)$, we denote by
$\phi_{\pi,\ell}$ the locally integrable function on $G(E)$
representing the distribution $\Phi_{\pi,\ell}$, that is,
$$\Phi_{\pi,\ell}(f)=\int_{G(E)}\phi_{\pi,\ell}(g)f(g)\ \d g,
\quad f\in C_c^\infty(G(E)).$$ We will also call $\phi_{\pi,\ell}$ a
{\em spherical character}. Note that $\phi_{\pi,\ell}$ is
bi-$G(F)$-invariant. Theorem \ref{thm. hakim} is analogous to the
classical result of Harish-Chandra on admissible invariant
distributions on connected reductive $p$-adic groups (see
\cite[Theorem 16.3]{hc1}).

When $\pi$ is unitary supercuspidal, we will show that the spherical
characters $\phi_{\pi,\ell}$ satisfy an orthogonality relation (see
Theorem \ref{thm. orthogonal}). Before stating this relation, we
need to review the Weyl integration formula in the setting of
symmetric pairs. We refer the reader to \cite[\S3]{rr} or
\cite[\S6]{ha2} for the notation below and more details on this
integration formula.

Let $\sT$ be a set of representatives for the equivalence classes of
Cartan subsets of $G(E)$ with respect to the involution $\theta$.
For $T\in\sT$, denote $T_\reg=T\cap G(E)_\reg$. For $T\in\sT$, the
map
$$\mu:G(F)\times T_\reg\times G(F)\lra G(E)_\reg,
\quad (h_1,t,h_2)\mapsto h_1th_2,$$ is submersive and
$$G(E)_\reg=\coprod_{T\in\sT}G(F)T_\reg G(F).$$
Let $A$ be the split component of the center of $G$. For each
$\theta$-regular element $g$, we choose a Haar measure on
$G_\gamma(F)$ where $\gamma=s(g)$ and $G_\gamma$ is the split
component of the centralizer of $\gamma$ in $G$. Fix Haar measures
on $A(F)$ and $G(F)$. For $\phi\in C_c^\infty(G(E)/A(F))$ and $g\in
G(E)_\reg$, the orbital integral $O(g,\phi)$ of $\phi$ at $g$ is
defined to be
$$O(g,\phi)=\int_{A(F)\bs G(F)}\int_{G_\gamma(F)\bs G(F)}\phi(h_1gh_2)
\ \d h_1\ \d h_2,$$ where $\gamma=s(g)$ and the measures inside the
integral are quotient measures. For $T\in\sT$, the group $G_{s(t)}$
is the same for each $t\in T_\reg$. Let $D_{G(E)}$ be the usual Weyl
discriminant function on $G(E)$. Then the Weyl integration formula
reads as follows: with suitably normalized measures, for each
$\phi\in C_c^\infty(G(E)/A(F))$, we have
$$\int_{G(E)/A(F)}\phi(g)\ \d g=\sum_{T\in\sT}\frac{1}{w_T}\int_{T}
\abs{D_{G(E)}(s(t))}_E\cdot O(t,\phi)\ \d t,$$ where $w_T$ are some
positive constants only depending on $T$ (cf. \cite[Theorem 3.4]{rr}
and \cite[Lemma 5]{ha2}). Let $\sT_\el$ be the subset of $\sT$
consisting of elliptic Cartan subsets, that is, for $T\in\sT$, $T$
belongs to $\sT_\el$ if and only if $T_\reg\subset G(E)_\el$.

\begin{thm}\label{thm. orthogonal}\begin{enumerate}
\item Suppose that $\pi$ is unitary supercuspidal and
$G$-distinguished. Let $\ell$ be a nonzero element of
$\Hom_{G(F)}(\pi,\BC)$. Then
$$\sum_{T\in\sT_\el}\frac{1}{w_T}\int_{T}\abs{D_{G(E)}(s(t))}_E\cdot\abs{\phi_{\pi,\ell}(t)}^2\
\d t$$ is nonzero.
\item Suppose that $\pi$ and $\pi'$ are two unitary supercuspidal
representations of $G(E)$ and $\pi\ncong\pi'$. Then for any
$\ell\in\Hom_{G(F)}(\pi,\BC)$ and $\ell'\in\Hom_{G(F)}(\pi',\BC)$,
we have
$$\sum_{T\in\sT_\el}\frac{1}{w_T}\int_{T}\abs{D_{G(E)}(s(t))}_E\cdot\phi_{\pi,\ell}(t)
\cdot\overline{\phi_{\pi',\ell'}(t)}\ \d t=0.$$
\end{enumerate}
\end{thm}

Theorem \ref{thm. orthogonal} is an analog of the classical
orthogonality relation for characters of discrete series (see
\cite{cl} or \cite{ka} for this classical result). The following
corollary, which has potential application in simple relative trace
formula, is a direct consequence of Theorem \ref{thm. orthogonal}.

\begin{cor}\label{cor. nonvanishing}
Suppose that $\pi$ is unitary supercuspidal and $G$-distinguished.
Let $\ell$ be a nonzero element of $\Hom_{G(F)}(\pi,\BC)$. Then the
spherical character $\Phi_{\pi,\ell}$ does not vanish on $G(E)_\el$.
\end{cor}

\paragraph{Acknowledgements.}
This work was partially supported by NSFC 11501033 and the
Fundamental Research Funds for the Central Universities. The author
thanks Wen-Wei Li for helpful discussions, especially for his help
on some arguments in the proof of Lemma \ref{lem. orbital integral
of supercusp form}.

\section{Proof of Theorem \ref{thm. orthogonal}}

\begin{lem}\label{lem. galois cohomology}
Suppose that $\gamma=s(g)$ with $g\in G(E)$ lies in an $F$-Levi
subgroup $M$ of $G$. Then there exists $m\in M(E)$ such that
$\gamma=s(m)$.
\end{lem}

\begin{proof}
First we recall some basic facts about symmetric spaces. Denote
$\FG=\R_{E/F}G$ and $\FM=\R_{E/F}M$. Let $\FX=\FG/G$ and
$\FX_M=\FM/M$ be the quotient varieties. As $F$-varieties, $\FX$ and
$\FX_M$ are isomorphic to the identity components of the varieties
defined by the equations
$$\tilde{\FX}=\{x\in\FG:\ x\theta(x)=1\}\quad\textrm{and}\quad
\tilde{\FX}_M=\{x\in\FM:\ x\theta(x)=1\}$$ respectively. The exact
sequences $$1\ra G\ra\FG\ra\FX\ra1\quad \textrm{and}\quad 1\ra
M\ra\FM\ra\FX_M\ra1$$ induce the following exact cohomology
sequences:
$$1\ra s(\FG(F))\ra\FX(F)\ra H^1(F,G)$$ and
$$1\ra s(\FM(F))\ra\FX_M(F)\ra H^1(F,M),$$ where we use the standard notation
$H^1(F,\bullet)$ to denote the Galois cohomology of algebraic groups
\cite[Chapter III. \S2]{ser}. However, the above exact sequences
have a little concern with our assertion. What we need are the
following exact sequences \cite[Lemma 4.1.1]{car}:
$$1\ra s(\FG(F))\ra\tilde{\FX}(F)\ra H^1(\theta,\FG(F))\ra 1$$
and
$$1\ra s(\FM(F))\ra\tilde{\FX}_M(F)\ra H^1(\theta,\FM(F))\ra 1,$$
where
$$H^1(\theta,\FG(F)):=H^1(\Gal(E/F),G(E))$$ and
$$H^1(\theta,\FM(F)):=H^1(\Gal(E/F),M(E)).$$

Note that $\gamma\in\tilde{\FX}_M(F)$ and Lemma \ref{lem. galois
cohomology} asserts that $\gamma\in s(\FM(F))$. According to
\cite[Corollary 4.1.2]{car}, it suffices to show that the image
$[\gamma]_M$ of $\gamma$ in $H^1(\theta,\FM(F))$ is trivial. On the
other hand, we know that the image $[\gamma]_G$ of $\gamma$ in
$H^1(\theta,\FG(F))$ is trivial, and $[\gamma]_G$ is also the image
of $[\gamma]_M$ under the natural map
$$\iota:H^1(\theta,\FM(F))\ra H^1(\theta,\FG(F)).$$ We claim that
$\iota$ is injective, which implies that $[\gamma]_M$ is trivial.
Consider the exact sequences \cite[Chapter I. \S5.8(a)]{ser}:
$$1\ra H^1(\theta,\FG(F))\ra H^1(F,G)\ra H^1(E,G)^{\Gal(E/F)}$$
and
$$1\ra H^1(\theta,\FM(F))\ra H^1(F,M)\ra H^1(E,M)^{\Gal(E/F)}.$$
Let $P$ an $F$-parabolic subgroup of $G$ such that $P=M\ltimes U$
where $U$ is the unipotent radical of $P$. We have natural
isomorphisms (see \cite[Lemma 16.2]{gil})
$$H^1(F,P)\stackrel{\simeq}{\lra} H^1(F,M),\quad\textrm{and}\quad
H^1(E,P)\stackrel{\simeq}{\lra} H^1(E,M),$$ and natural injections
\cite[Chapter III. \S2.1]{ser}
$$H^1(F,P)\incl H^1(F,G)\quad\textrm{and}\quad
H^1(E,P)\incl H^1(E,G).$$ In summary we have the following
commutative diagram of exact sequences:
$$\begin{matrix}
&1&\ra&H^1(\theta,\FG(F))&\ra&H^1(F,G)&\ra&H^1(E,G)\\
&&&\uparrow&&\uparrow&&\uparrow\\
&1&\ra&H^1(\theta,\FP(F))&\ra&H^1(F,P)&\ra&H^1(E,P)\\
&&&\downarrow&&\downarrow&&\downarrow\\
&1&\ra&H^1(\theta,\FM(F))&\ra&H^1(F,M)&\ra&H^1(E,M),\\
\end{matrix}$$
which implies that $\iota$ is injective.

\end{proof}

\begin{lem}\label{lem. orbital integral of supercusp form}
Suppose that $\phi$ is a matrix coefficient of a unitary
supercuspidal $G$-distinguished representation. Then, for any $g\in
G(E)_\reg$, the orbital integral $O(g,\phi)$ vanishes unless $g$ is
$\theta$-elliptic.
\end{lem}
\begin{proof}
Since $\phi$ is a matrix coefficient of a unitary supercuspidal
$G$-distinguished representation, it belongs to
$C_c^\infty(G(E)/A(F))$ and is a supercusp form. In particular, for
any unipotent radical $N$ of a proper parabolic subgroup $P$ of $G$,
we have
$$\int_{N(E)}\phi(gn)\ \d n=0$$ for any $g\in G(E)$.
Write $\gamma=s(g)$. Suppose that $g$ is not $\theta$-elliptic,
which means that $\gamma$ is not elliptic by definition. Therefore
there exists a Levi subgroup $M$ of a proper parabolic subgroup $P$
of $G$ such that $G_\gamma\subset M$. According to Lemma \ref{lem.
galois cohomology} there exists $m\in M(E)$ such that $\gamma=s(m)$.
Since
$$O(g,\phi)=O(m,\phi),$$  we assume that $g$ is in
$M(E)$ from now on. Let $N$ be the unipotent radical of $P$, and $K$
a maximal open compact subgroup of $G(F)$ such that
$G(F)=M(F)N(F)K$. Fix Haar measures $\d m,\ \d n$ and $\d k$ on
$M(F)/A(F),\ N(F)$ and $K/K\cap A(F)$ respectively so that $\d h=\d
k\ \d n\ \d m$ on $G(F)/A(F)$. Denote $\bar{K}=K/K\cap A(F)$. Then
the orbital integral $O(g,\phi)$ can be written as follows:
$$\begin{aligned}
O(g,\phi)&=\int_{A(F)\bs G(F)}\int_{G_\gamma(F)\bs
G(F)}\phi(h^{-1}_1gh_2)\ \d
h_2\ \d h_1\\
&=\int_{\left(A(F)\bs M(F)\right)\times N(F)\times
\bar{K}}\int_{\left(G_\gamma(F)\bs M(F)\right)\times N(F)\times
\bar{K}}\phi(k_1^{-1}n_1^{-1}m_1^{-1}gm_2n_2k_2)\\
&\ \ \cdot\ \d k_1\ \d k_2\ \d n_1\ \d n_2\ \d m_1\ \d m_2\\
&=\int_{\left(A(F)\bs M(F)\right)\times
N(F)}\int_{\left(G_\gamma(F)\bs M(F)\right)\times
N(F)}\phi'(n_1^{-1}m_1^{-1}gm_2n_2)\\
&\ \ \cdot\ \d n_1\ \d n_2\ \d m_1\ \d m_2,
\end{aligned}$$
where $$\phi'(x):=\int_{\bar{K}\times \bar{K}}\phi(k_1xk_2)\ \d k_1\
\d k_2,\quad x\in G(E)/A(F).$$ Note that $\phi'$ is still a
supercusp form on $G(E)$. From now on, for convenience, we write
$\phi$ instead of $\phi'$ and $g$ instead of $m_1^{-1}gm_2$. Let
$\gamma=s(g)$ for this ``new'' $g$. We claim that:
\begin{equation}\label{equ. vanishing of the inner integral}
\int_{N(F)\times N(F)}\phi(n_1^{-1}gn_2)\ \d n_1\ \d n_2=0.
\end{equation}
It is clear that this claim implies the lemma directly.

Now we begin to prove claim (\ref{equ. vanishing of the inner
integral}). Note that $$\int_{N(F)\times N(F)}\phi(n_1^{-1}gn_2)\ \d
n_1\ \d n_2=\int_{N(F)\times N(F)}\phi(g\cdot g^{-1}n_1gn_2)\ \d
n_1\ \d n_2.$$ Denote $\FN=\R_{E/F}N$. Consider the morphism of the
algebraic varieties:
$$\eta_g:N\times N\lra \FN,\quad (n_1,n_2)\mapsto
g^{-1}n_1gn_2.$$ We will show that $\eta_g$ is an isomorphism. If
$g^{-1}n_1gn_2=g^{-1}n_1'gn_2'$, we have the relation
$$n_2^{-1}\gamma n_2=s(n_1gn_2)=s(n'_1gn_2')=n_2'^{-1}\gamma n'_2.$$
By \cite[Lemma 22]{hc}, this relation implies $n_2=n'_2$, and thus
$n_1=n_1'$. Hence $\eta_g$ is injective. To show $\eta_g$ is
surjective, consider the Lie algebras $\fn'=\Lie(N')$,
$\fn''=\Lie(N)$ and $\fn=\Lie(\FN)$, where $N'$ is the unipotent
subgroup $g^{-1}Ng$. Since
$$2\dim_F\fn'=2\dim_F\fn''=\dim_F\fn$$ and $\fn'\cap\fn''=\{0\}$ by
the injectivity of $\eta_g$, we have $\fn=\fn'\oplus\fn''$.
Therefore $\eta_g$ is submersive and thus $N'\cdot N$ is open in
$\FN$. On the other hand, since $N'$ and $N$ are unipotent groups,
the orbit $N'\cdot N$ of $1$ under the left and right translations
of $N'$ and $N$ is closed in $\FN$. Hence $\FN=N'\cdot N$, that is,
$\eta_g$ is surjective. It turns out that
$$\int_{N(E)}\phi(gn)\ \d n
=\int_{N(F)\times N(F)}j_g(n_1,n_2)\cdot\phi(g\cdot g^{-1}n_1gn_2)\
\d n_1\ \d n_2,$$ where $j_g(n_1,n_2)$ is the Jacobian of $\eta_g$
at $(n_1,n_2)$. Note that
$$j_g(n_1,n_2)=\abs{\ad(g)|_{\fn(F)}}_E,$$ which is independent of
$(n_1,n_2)$. At last, the claim (\ref{equ. vanishing of the inner
integral}) follows from the condition that $\phi$ is a supercusp
form.

\end{proof}

\begin{proof}[Proof of Theorem \ref{thm. orthogonal}]
By \cite[Theorem 1.5]{zh}, there exists a vector $u_0$ in the space
$V_\pi$ such that $\ell=\sL_{u_0}$, where the $G(F)$-invariant
linear form $\sL_{u_0}$ is defined by
$$\sL_{u_0}(v):=\int_{G(F)/A(F)}\pair{\pi(h)v,u_0}\ \d h,\quad v\in
V_\pi.$$ Set $$\phi(g)=\pair{\pi(g)u_0,u_0},$$ which is a matrix
coefficient of $\pi$. Then, according to \cite[Corollary 1.11]{zh},
the spherical character $\Phi_{\pi,\ell}$ has the following
expression
\begin{equation}\label{equ. spherical character expression}
\Phi_{\pi,\ell}(f)=\int_{G(F)/A(F)}\int_{G(F)/A(F)}\int_{G(E)}\phi(h_1gh_2)f(g)
\ \d g\ \d h_1\ \d h_2.\end{equation} In particular, when $f\in
C_c^\infty(G(E)_\el)$, we get
$$\Phi_{\pi,\ell}(f)=\int_{G(E)}O(g,\phi)f(g)\ \d g.$$ On the other
hand, by Theorem \ref{thm. hakim}, we have
$$\Phi_{\pi,\ell}(f)=\int_{G(E)}\phi_{\pi,\ell}(g)f(g)\ \d g.$$
Therefore, for $g\in G(E)_\el$, we obtain
\begin{equation}\label{equ. character and orbital integral}
\phi_{\pi,\ell}(g)=O(g,\phi).\end{equation}

For the first assertion, choose $f_1\in C_c^\infty(G(E))$ so that
$\phi_{f_1}=\bar{\phi}$, where $$\phi_{f_1}(g):=\int_{A(F)}f_1(ag)\
\d a.$$ Then
$$\begin{aligned}\Phi_{\pi,\ell}(f_1)&=\int_{G(E)/A(F)}\phi_{\pi,\ell}(g)\cdot\phi_{f_1}(g)
\ \d g\\
&=\sum_{T\in\sT}\frac{1}{w_T}\int_{T}\abs{D_{G(E)}(s(t))}_E\cdot
\phi_{\pi,\ell}(t)\cdot O(t,\bar{\phi})\ \d t.
\end{aligned}$$
Combining Lemma \ref{lem. orbital integral of supercusp form} and
(\ref{equ. character and orbital integral}), we get
\begin{equation}\label{equ. spherical character elliptic locus}
\Phi_{\pi,\ell}(f_1)=\sum_{T\in\sT_\el}\frac{1}{w_T}\int_{T}\abs{D_{G(E)}(s(t))}_E\cdot
\abs{\phi_{\pi,\ell}(t)}^2\ \d t.\end{equation} Let
$$v_0=\frac{1}{\sqrt{\pair{u_0,u_0}}}u_0,$$ and choose
$\{v_i\}_{i\in\BN}$ such that $\{v_i\}_{i\geq0}$ is an orthonormal
basis of $V_\pi$. Then $$\pi(\bar{\phi})v_0=\lambda v_0\ \textrm{for
some nonzero }\lambda,\quad \textrm{and }\pi(\bar{\phi})v_i=0\
\textrm{for }i\geq1.$$ Therefore
$\Phi_{\pi,\ell}(f_1)=\lambda\abs{\ell(v_0)}^2$. From the proof of
\cite[Theorem 1.5]{zh}, we see that
$$\overline{\ell(u_0)}=c\int_{A(E)G(F)\bs G(E)}\abs{\ell\left(\pi(g)u_0\right)}^2\ \d g
=c'\pair{u_0,u_0},$$ where $c$ and $c'$ are some nonzero numbers.
Hence $\Phi_{\pi,\ell}(f_1)$ is nonzero. We complete the proof of
the first assertion.

For the second assertion, choose $f_2\in C_c^\infty(G(E))$ in
(\ref{equ. spherical character expression}) so that
$\phi_{f_2}=\bar{\phi'}$, where $\phi'$ is a matrix coefficient of
$\pi'$ such that the distribution $\Phi_{\pi',\ell'}$ can be
expressed as
$$\Phi_{\pi',\ell'}(f)=\int_{G(F)/A(F)}\int_{G(F)/A(F)}\int_{G(E)}\phi'(h_1gh_2)f(g)
\ \d g\ \d h_1\ \d h_2$$ for any $f\in C_c^\infty(G(E))$. By the
same reason for the equality (\ref{equ. spherical character elliptic
locus}), we have
$$\Phi_{\pi,\ell}(f_2)=\sum_{T\in\sT_\el}\frac{1}{w_T}\int_{T}\abs{D_{G(E)}(s(t))}_E\cdot
\phi_{\pi,\ell}(t)\cdot \overline{\phi_{\pi',\ell'}(t)}\ \d t.$$ On
the other hand
$$\Phi_{\pi,\ell}(f_2)
=\int_{G(F)/A(F)}\int_{G(F)/A(F)}\int_{G(E)/A(F)}\phi(h_1gh_2)\bar{\phi'}(g)
\ \d g\ \d h_1\ \d h_2=0,$$ since the inner integral over
$G(E)/A(F)$ vanishes by the Schur orthogonality relation. We
complete the proof of the second assertion.
\end{proof}

\s{\small Chong Zhang\\
School of Mathematical Sciences, Beijing Normal University,\\
Beijing 100875, P. R. China.\\
E-mail address: \texttt{zhangchong@bnu.edu.cn}}

\end{document}